\documentclass[a4paper,12pt]{amsart}
\usepackage{hyperref,latexsym}
\usepackage{enumerate}

\theoremstyle{plain}
\newtheorem{theorem}{Theorem}[section]
\newtheorem{lemma}[theorem]{Lemma}
\newtheorem{corollary}[theorem]{Corollary}

\theoremstyle{definition}
\newtheorem{definition}[theorem]{Definition}
\newtheorem{example}{Example}
\theoremstyle{remark}

\begin{document}

\title[ Integrable
geodesic flows on surfaces]
      {Integrable geodesic flows on surfaces}

\date{9 May 2009}
\author{Misha Bialy}
\address{Raymond and Beverly Sackler School of Mathematical Sciences, Tel Aviv University,
Israel} \email{bialy@post.tau.ac.il}


\begin{abstract}
We propose a new condition $\aleph$ which enables to get new results
on integrable geodesic flows on closed surfaces. This paper has two
parts. In the first, we strengthen Kozlov's theorem on
non-integrability on surfaces of higher genus. In the second, we
study integrable geodesic flows on 2-torus. Our main result for
2-torus describes the phase portraits of integrable flows. We prove
that they are essentially standard outside, what we call, separatrix
chains. The complement to the union of the separatrix chains is
$C^0$-foliated by invariant sections of the bundle.
\end{abstract}

\maketitle

\section{Introduction}
\label{sec:intro} Let $\Sigma$ be a closed orientable surface of
genus $p\geq1$. Given a Riemannian metric $g$ on $\Sigma$, let
$g^t:T_1 \Sigma \rightarrow T_1 \Sigma$ be the corresponding
geodesic flow acting on the unite circle bundle of $\Sigma$. It is
an important question for dynamics and geometry if there exists a
smooth function $F:T_1\Sigma\rightarrow\mathbb{R}$ which is
invariant under the flow, that is $F(g^tx)=F(x)$. In this case $F$
is classically called a (first)integral of the geodesic flow. In
this case, Liouville-Arnold theorem \cite{Arn} implies that any
connected component $L$ of the level set $\{F=c\}$, which satisfies
$DF|_L\neq 0$, is a 2-torus invariant under the flow. Moreover, the
dynamics of geodesic flow on this torus is linearizable. We shall
use the following:
\begin{definition}A torus $L$ lying in a level of $F$ we will call
regular, if $DF|_L\neq 0$, on the other hand $L$ is singular if
there exists a point on $L$ where $DF$ vanishes.
\end{definition}

In the search of smooth integrals, one usually requires, some extra
condition on $F$ which prohibits from $F$ to be essentially
constant. For example, one usually requires the set of regular
points of $F$ to be dense. In this case the geodesic flow is called
integrable.

The following theorem was proved in 1979 by V.V. Kozlov see \cite
{Kozl} and also \cite{Koz}.
\begin{theorem}
\label{kozlov} Assume the genus $p>1$. Then any smooth integral
$F:T_1\Sigma\rightarrow\mathbb{R}$ satisfying the following two
conditions must be a constant.

    (K1) $F$ has finitely many critical values.

    (K2) For a dense set of points $x\in\Sigma$ the intersection of
    the fibre $\pi^{-1}(x)$ with any critical level $\{F=c\}$ is at
    most finite or coincides with the whole fibre.

\end {theorem}
 As a corollary Kozlov concluded that any
 real-analytic integral of the geodesic flow on
 surfaces of higher genus must be constant. The question on topological
obstructions to the integrability of geodesic flows has found
considerable interest, see e.g. the works by \cite{Tai1},
\cite{Tai2}, \cite{Pat1}, \cite{Pat2} for generalizations to higher
dimensions where various stronger assumptions on non-wildness of $F$
were proposed. On the other hand, new remarkable examples of
manifolds which cannot have analytically integrable geodesic flows,
but do admit a $C^\infty$-integrable geodesic flows were recently
discovered, see the works \cite {Bol},\cite{But1},\cite{But2}.
However, it seems that the question if such examples exist on
compact surfaces of genus greater than one is still open.

We approach this question with geometric idea which enables to relax
significantly the conditions of Kozlov's theorem. On the other hand,
we apply this method to the case of the 2-torus. For the 2-torus
there are two known classes of metrics with integrable geodesic
flows. These are rotationally symmetric metrics and, the so-called,
Liouville metrics. The question if there exist other examples of
integrable geodesic flows is widely open. By our method we get a
nontrivial information on the phase portraits for integrable
geodesic flows. We show that they are essentially standard outside,
what we call separatrix chains. Namely the complement to the union
of these chains is $C^0$ foliated by invariant sections. So all
dynamical complications could be located only inside the chains.

The first ingredient of the method is the theory of minimal
geodesics and rays on surfaces. It was invented by M.Morse
\cite{Morse} and G.Hedlund \cite{Hed} and further treated in
\cite{BP1} and also in \cite{Ban1} in connection with Aubry-Mather
theory. The second ingredient uses the properties of the projections
of Lagrangian torii started in \cite{BP2}, \cite{B2}, \cite{P} (we
refer to \cite{BP3} for generalizations and the references).

For the approach of this paper we shall require everywhere that the
metric $g$ and the function $F$ to be of class $C^3$, at least. This
is the minimal regularity needed in order to use the properties of
projections for Lagrangian torii and also for the Morse-Sard theorem
for $F$ which is used below.

In order to state the first result let me formulate the main
condition.
\begin{definition}
We shall say that $F:T_1\Sigma\rightarrow\mathbb{R}$ satisfies
\emph{condition $\aleph$}, if for a dense set of $x\in \Sigma$ the
intersection of
    the fibre $\pi^{-1}(x)$ with any connected component of the
    critical level $\{F=c\}$ is at
    most \emph{countable} or coincides with the whole fibre.
\end{definition}
\begin{example}
Any function which is real-analytic with respect to momenta
satisfies this condition. In particular polynomials with respect to
momenta are of great interest, since in all known examples of
integrable geodesic flows the integral appears to be polynomial in
momenta variables.
\end{example}

\begin {theorem}
\label{genus} Suppose genus $p>1$. Then any integral $F$ of the
geodesic flow satisfying condition $\aleph$ must be a constant.
\end {theorem}
\begin {corollary} There are no real-analytic with respect to
momenta integrals for geodesic flows on surface of higher genus
other than constants.
\end {corollary}
Let me mention that the first condition $(K1)$ of Kozlov's theorem
is not required in theorem \ref{genus}.

Our next results apply for the case of Riemannian 2-torus
$\mathbb{T}^2=\mathbb{R}^2/\mathbb{Z}^2$.
\begin{theorem}
\label{torus}
 Let $\Sigma=\mathbb{T}^2$. Let $F$ be a \emph{non-constant} integral of
the geodesic flow satisfying condition $\aleph$. Then there exists
an invariant 2-torus $L$ of the geodesic flow lying in the regular
level of $F$ which is a smooth section of $T_1\mathbb{T}^2$.
\end{theorem}
In fact we can specify in the following way:
\begin{corollary}
\label{rational}
 The torus $L$ in the theorem can be chosen
in such a way that all the orbits of the flow on $L$ project to
minimal closed geodesics of the same homotopy type.
\end{corollary}

One can deepen the theorem \ref{torus} in the following way, which
can be interpreted as non-existence of the "instability" zones:

\begin{theorem}
\label{zones} Let $F$ be a non-constant integral of the geodesic
flow satisfying condition $\aleph$. Let  $N$ be a domain in
$T_1\mathbb{T}^2$ bounded by two disjoint sections $L_1,L_2$ which
are singular invariant torii of the geodesic flow. Then there exists
a regular torus $L$ lying inside $N$ which is a section invariant
under the flow.
\end{theorem}

Notice that the claim of this theorem obviously holds true if one of
the boundary torii $L_1,L_2$ is regular. Because one can push it
inside $N$ by a flow $v^t$ of the vector field $v=\nabla F /|\nabla
F|^2$ (see proof of theorem \ref {chains} below). Moreover, assuming
theorem \ref{torus} one can start the following process. Take a
regular torus $L$ and move it with the flow $v^t$ inside
$T_1\mathbb{T}^2$ as long as possible in a positive and negative
direction. This cannot be continued when the torii become singular.
Then, by theorem \ref{zones} between any two singular torii one can
take a new regular torus and to flow it again, and so on. The
process terminates when one gets singular torii touching one
another. To be more precise, we introduce the following
\begin{definition} By a separatrix chain we mean a closed subset
$X \subset T_1\mathbb{T}^2$ bounded by two different singular torii
$L_+,L_-$ such that both of them are Lipshits sections invariant
under the geodesic flow such that the intersection $L_{+}\cap L_-$
equals the union of all periodic minimizing trajectories of a common
rational rotation number. No other invariant sections are allowed to
pass inside $X$.
\end{definition}

\begin{theorem}
\label{chains} Let $F$ be a non-constant integral of the geodesic
flow of the 2-torus satisfying condition $\aleph$. There are at most
countably many separatrix chains. Through any point in the
complement of their union passes a unique invariant section which is
either regular or singular. In case it is singular, it is a limit
from both sides of regular invariant sections. The complement to the
union of separatrix chains is $C^0$ foliated by invariant sections.
\end {theorem}
For example, there are no separatrix chains at all if the metric is
flat (and in fact only in this case, by a theorem of E.Hopf
\cite{Hopf}). There are precisely two chains for rotationally
symmetric Riemannian metrics and four of them for Liouville metrics.
The number of these separatrix chains (it is always even, due to the
symmetry of the metric) corresponds to a number of non-smooth points
for the ball of stable norm on $H_1(\mathbb{T}^2;\mathbb{R})$. This
connection with the stable norm follows from Bangert's paper
\cite{Ban2} (see also paper by Mather \cite{Mat}).
\begin{corollary}
If the integral $F$ is assumed to be real-analytic in momenta then
there are at most finitely many of separatrix chains.
\end{corollary}
This is because the  function $F$ has equal values on $L_+$ and
$L_-$ and therefore the derivative along the fibre of $F$ must
vanish  somewhere in between. By analyticity this can happen only
finitely many times.

The organization of the paper is as follows. In Section
\ref{background} we summarize the needed facts on minimal geodesics
and Lagrangian projections of invariant torii. In Section
\ref{interior} we prove that no minimal rays can be trapped "inside"
compressible invariant tori. Section \ref{proofs} contains the
proofs of main theorems \ref {genus}, \ref {torus}. In Section \ref
{separatrix} theorems \ref {zones} and \ref {chains} are proved. The
last Section \ref{critical} contains facts on critical points of $F$
(not necessarily satisfying condition $\aleph$).


\section*{Acknowledgements}
I am thankful to Leonid Polterovich, with whom we started the study
of Lagrangian torii and Minimal geodesics many years ago. During
these years we discussed the subject many times and I learnt a lot
of things from him.

\section {Minimal geodesics and torii}
\label{background} The proof of our main results relies  on
interplay between theory of minimal geodesics on surfaces and theory
of Lagrangian singularities of projections of invariant torii. Let
me summarize the needed facts in several theorems below.

Let us represent the surface as a quotient $\Sigma=\widetilde{\Sigma
}/\Gamma$ where $\Gamma$ is the fundamental group of the surface
acting on the covering by isometries. Here $\widetilde{\Sigma }$ is
a Euclidian plane $\mathbb{R}^2$ for $p=1$, and the Poincare unite
disc $\mathbb{D}$ for $p>1$. We shall lift the metric $g$ to the
covering and also denote by $g_{0}$ the Euclidian and Hyperbolic
metric on $\widetilde{\Sigma }$ respectively. We use the following
terminology. An isometric image of $
\gamma:(-\infty,+\infty)\rightarrow\widetilde{\Sigma }$ will be
called a minimal geodesic while an isometric image of $
\gamma:[0,+\infty)\rightarrow\widetilde{\Sigma }$ will be called a
ray starting at $\gamma(0)$. Projections of minimal geodesics and
rays from $\widetilde{\Sigma }$ to $\Sigma$ form by definition a
class of minimal geodesics and rays on $\Sigma$. The orbits of
geodesic flow corresponding to minimal geodesics and rays we will
call by minimal orbits. It was proved by Morse that each minimal
geodesic stays in a finite distance from a unique $g_0$-geodesic on
the covering $\widetilde{\Sigma }$ which is called the type of the
minimal geodesic. The type is determined by the slope of the
straight line on the Euclidian plane for $p=1$, and by the end
points of the Hyperbolic geodesic for $p>1$. A minimal geodesic is
called to be of periodic type if the corresponding $g_0$-geodesic
becomes closed being projected to $\Sigma$, and it has a
non-periodic type, in the opposite case. We shall use the following:
\begin{theorem}
\label{rays}
\item
1. For any point $x\in\ \widetilde{\Sigma }$ there are uncountably
many rays starting at $x$. They are parameterized by the slope for
the case $p=1$, and by points of the ideal circle in the case $p>1$.
This means that for the case $p=1$, every such a ray has a given
slope. And in the case $p>1$ such a ray stays on a bounded distance
at infinity from any Hyperbolic geodesic approaching a given point
of the ideal circle.
\item
2. Some of these rays are of periodic type. More precisely, for any
point $x\in \widetilde{\Sigma }$ there exists a ray starting at $x$
asymptotic to a minimal geodesic of a given periodic type (or the
ray itself is a half of periodic minimal geodesic).

\end{theorem}

The second statement of the theorem is proved explicitly in the
paper by M.Morse. The first is immediate for the case of the torus
(see also \cite {BP1}) and for higher genus it can be easily deduced
from the theory of minimal geodesics in the following way. Given a
point $x\in\widetilde{\Sigma}$ and a point $y$ on the ideal
boundary, choose hyperbolic geodesics $l_1,l_2$ with one their end
at $y$, such that $x$ is contained in the strip between them and
lies sufficiently far away from them. Then one can construct by a
limiting argument, two minimal $g$-geodesics $\gamma_1,\gamma_2$ of
the type of $l_1,l_2$ respectively such that they do not intersect
and contain $x$ in the strip between them. Having such a strip, one
constructs a ray lying inside in a standard way. Notice that unlike
the torus case where the ray can be chosen to be asymptotic to a
minimal geodesic of given slope, the ray constructed above for
higher genus stays on a bounded distance at infinity from
$\gamma_1,\gamma_2$. The following fact, important for our purposes,
was proved by Morse \cite{Morse}, for $p>1$, and by Hedlund in
\cite{Hed}, for $p=1$.  Given any two minimal periodic geodesics of
the same periodic type on $\widetilde{\Sigma}$ such that there are
no other minimal periodic geodesics in the strip between them, there
always exist two heteroclinic connecting geodesics in the strip
between them (and therefore also all their translates). It follows
from this fact that the set of minimal geodesics of a given periodic
type on $\widetilde{\Sigma}$ has separation zero, i.e. for any
$\epsilon>0$ there exists an $\epsilon$-chain connecting them. We
can summarize:
\begin{theorem}
\label{connectedness} Fix a periodic type of minimal geodesics on
$\widetilde{\Sigma}$ and denote  by $\widetilde{M}\subset
T_1\widetilde{\Sigma}$ be a set of all unite tangent vectors to
minimal geodesics of the fixed periodic type, let $M$ be the
projection of $\widetilde{M}$ to $T_1\Sigma$. Then
\item
1. The set $\widetilde{M}$ has separation zero, $M$ is a closed
connected set.
\item
2. For any point $x\in\widetilde{\Sigma}$ there exists a unite
tangent vector $v$ at $x$ such that the trajectory of the geodesic
flow $g^t(x,v)$ is approaching $\widetilde{M}$ as
$t\rightarrow+\infty$.
\end{theorem}
Let me summarize now the needed facts on Lagrangian torii invariant
under the geodesic flows. These results started from \cite{BP1},
\cite{BP2} generalizing Birkhoff's first and second theorems from
area preserving twist maps to the case of geodesic flows. Later they
were generalized further in many other directions. We refer the
reader to paper \cite{BP3} for various generalizations and
references. I will need also the results from \cite{B2} where
compressible invariant torii were studied. I will remind first the
following:
\begin{theorem}
\label{caustics} Let $L\subset T_1\Sigma$ be an invariant torus of
the geodesic flow. Then the set of singular points of the projection
of $\pi|_L$ is a union of finite number simple closed
non-intersecting curves on $L$. They are not null-homotopic on $L$
and the trajectories of the flow intersect them transversally.
\end{theorem}
 Recall that an imbedded torus $L$ inside $T_1\Sigma$ is
 called \emph{compressible} if the homomorphism induced by
 inclusion,
 $i_*:\pi_1(L)\rightarrow\pi_1(T_1\Sigma)$
 has a nontrivial kernel, and \emph{incompressible} otherwise.
 Remarkably, if $L$ is an invariant torus of the geodesic flow
 it happens that it is either compressible or projects diffeomorhically.
 Moreover, one can construct the compressing disc explicitly as it
 is
 proved in \cite{B2}.

  Let me summarize in the following way
 the known results on Lagrangian torii from \cite{BP2}, \cite{B2}.
\begin{theorem}
\label{birkhoff}Let $L \subset T_1\Sigma$ be a an invariant torus of
the geodesic flow.
\item
1. For $p>1$ all invariant torii $L\subset T_1\Sigma$ are
compressible. Incompressible torii may exist only for
$\Sigma=\mathbb{T}^2$. In this case, if the dynamics on $L$ is
chain-recurrent (this always holds for regular torii) then $L$ is a
smooth section of the $T_1\mathbb{T}^2$.
\item
2. Moreover, write $\mathbb{T}^2=\mathbb{R}^2/\mathbb{Z}^2$ and fix
a Euclidian structure on $\mathbb{R}^2$ in order to trivialize the
bundle $T_1\mathbb{T}^2=\mathbb{T}^2\times \mathbb{S}^1$. Then the
the projection of $L$ on the $\mathbb{S}^1$-factor is null
homotopic. So, $L$ is a graph of a smooth function $f:
\mathbb{T}^2\rightarrow\mathbb{S}^1$.
\item
3. There exists a constant $K>0$ depending only on the metric $g$
such that any continuous graph invariant under the flow is in fact
$K$-Lipshits. In particular the functions of the previous item have
a-priori bounded gradients: $|\nabla f|\leq K$. In addition,
trajectories on the graph projects to minimal geodesics.
\end{theorem}

The first statement of this theorem is a combination of the
so-called generalized second Birkhoff theorem for the case of
geodesic flows together with a general idea on sections for
invariant torii of geodesic flows from \cite{B2}. The third
statement is the so-called generalized first Birkhoff theorem, which
reflects the so called twist condition, together with the property
of field of extremals. The second statement follows from a little
topological argument using minimal geodesics of the periodic type.
Let me remark here that it was proved in \cite{B1} that in the
generalized Birkhoff theorem it is enough to assume for $L$ to be
only continuous, provided there are no periodic orbits on $L$.

\section{Non-trapping of minimal geodesics.}
\label{interior} Now I am in position to define the "interior"
components of compressible torii as follows. Let $L\subset
T_1\Sigma$ be a smooth invariant compressible torus of the geodesic
flow. Then either $L$ bounds a solid torus or $L$ is contained in a
a part of $T_1\Sigma$ homeomorphic to a 3-ball (it is proved with
the help of compressing disc, see \cite{Hatcher}). Define
accordingly the "interior" $I(L)$ to be the interior of the solid
torus bounded by $L$, or the interior of the component lying inside
the 3-ball. Notice that this is a correct definition. Indeed, if it
happened that $L$ bounds a solid torus and some component of the
complement lies in a ball, then it should be a solid torus lying in
the 3-ball. Because otherwise $T_1 \Sigma$ would be simply
connected, which is not the case. Here is one of our crucial
observations:
\begin{theorem}
\label {I(L)} Let $L$ be a smooth compressible invariant torus in
$T_1\Sigma$ such that the dynamics on it is chain-recurrent (for
example this always holds for regular  invariant torii). Then all
geodesic trajectories corresponding to minimal geodesics and rays do
not lie neither on $L$ nor in $I(L)$.
\end{theorem}
\begin{proof}
Let us prove first that no minimal trajectory or ray can lie on $L$.
We have $\pi|_L$ is not a diffeomorphism and due to the condition of
chain-recurrence we have two possible cases for the dynamics on $L$.
In the first case the orbits are not closed and therefore each orbit
must intersect the singularity curves of theorem \ref{caustics}
infinitely many times. Each such intersection increases by one the
Morse index of the corresponding geodesic. Thus it can not be
minimal. In the second case, all the trajectories on $L$ are closed,
then all of them project to geodesics of the same length, since the
torus $L$ is Lagrangian (as a torus in the (co-) tangent bundle).
Some of them necessarily intersect singularity curves and therefore
not minimal as explained above. But all of them have the same
length, therefore all the orbits on $L$ are not minimal.

It is simple to rule out the case when $I(L)$ lies in a ball, since
it cannot happen for minimal trajectory to lie in a ball. Because
otherwise any lift of the minimal geodesic to the universal cover
would be bounded on $\widetilde{\Sigma}$ which is impossible.

Suppose now that $I(L)$ is a solid torus. Assume by contradiction
that there exists a minimal half-trajectory inside $I(L)$ which
corresponds to a ray. Then there exists a point in $\omega$-limit
set of this half-trajectory. The orbit of this point is a minimal
orbit (by a limiting argument). This minimal orbit must lie also
inside $I(L)$ because as was proved above, it can not lie on the
boundary. In addition, this orbit must be of a periodic type,
determined by the core of the solid torus. Then by the first
statement of theorem \ref {connectedness} all minimal orbits of this
periodic type must stay inside $I(L)$ and by the second statement,
fiber of any point of the covering $\widetilde{\Sigma}$ can be
connected by a minimal trajectory to one of them lying inside. But
this is impossible since $L$ is invariant and can not be crossed.
\end{proof}
\section{Proof of main theorems}
\label{proofs} In this section we prove the theorems \ref {genus}
and then \ref {torus}.

\begin{proof}[Proof of theorem \ref {genus}]
The idea is very simple. If $\Sigma$ is of genus $p>1$ then it
follows from theorem \ref{birkhoff} that all regular torii are
compressible. Therefore, by theorem \ref {I(L)} all minimal
geodesics and rays lie outside their "interior" components and
belong to singular levels of the integral $F$. Denote by $R$ the
class of all regular torii and define the set
$$ Z=T_1\Sigma-\bigcup_{L\in R}I(L).$$
This is a compact invariant set.
\begin {lemma}
\label{Z} The set $Z$ is an invariant continuum, i.e. it is compact,
connected invariant set.
\end{lemma}
We complete first the proof of the theorem. Notice that by the very
construction, there are no regular torii left in $Z$ since deleting
$I(L)$ for all regular $L$ we delete, of course, $L$ itself since
nearby torii are regular either. So in particular all regular levels
of $F$ are deleted. Thus, $Z$ lies in a union of all singular
levels. So $F$ attains only critical values on $Z$. By Morse-Sard
theorem (proved in this case by A.Morse \cite {Antony}) for $F$ and
by the lemma it follows that $F$ on $Z$ attains a unique critical
value. Moreover, the set of all tangent vectors to all minimal
geodesics and rays lie in $Z$ by theorem \ref {I(L)}. By theorem
\ref {rays} there are uncountably many of them in a fibre of any
point $x\in\Sigma$. By condition $\aleph$ this forces $F$ to be a
constant. This completes the proof of theorem \ref {genus}.
\end{proof}
\begin{proof}[Proof of lemma \ref {Z}]
Since the unite cotangent bundle $T_1\Sigma$ is second-countable
topological space then, the union $\bigcup_{L\in R}I(L)$ can be
replaced by at most countable union. Therefore there exist a
sequence of torii $L_n\in R,n\geq 1$ such that $$\bigcup_{L\in
R}I(L)=\bigcup_{j=1}^\infty I(L_{j}).$$
 Then one can write
$$Z=T_1\Sigma-\bigcup_{j=1}^\infty I(L_{j})=\bigcap_{n=1}^\infty Z_n,$$
where $$Z_n=T_1\Sigma-\bigcup_{j=1}^n I(L_{j}).$$ It is clear that
each $Z_n$ is compact and connected and $Z_{n+1}\subseteq Z_n$.
Thus, $Z$ is compact and connected either as an intersection of
nested sequence of compact connected sets. This proves the lemma.
\end {proof}

\begin{proof}[Proof of theorem \ref {torus} and corollary \ref{rational}]
We claim first that there exists a torus $L$ in $T_1\mathbb{T}^2$
which lies in a regular level of $F$ and is not compressible.
Indeed, otherwise all torii of regular levels $F$ would be
compressible. Then one can can define the set $Z$ and proceed
exactly as in the proof of theorem \ref{genus}, getting a
contradiction with the assumption that $F$ is not a constant
function. So incompressible $L$ exists. Apply theorem \ref{birkhoff}
in order to get that $L$ can be written as a graph of a smooth
function $f:\mathbb{T}^2\rightarrow\mathbb{S}^1$ and all the orbits
on $L$ are minimal geodesics. There are two cases which may occur
for the dynamics on $L$ (recall $L$ is satisfies Liouville-Arnold
theorem). In the first case all the orbits are closed and then we
are done. In the second case geodesics of $L$ have irrational
rotation number. Since $L$ lies in a regular level one can take
torus $\tilde{L}$ lying in a a nearby levels which are also
invariant graphs consisting of minimal orbits. We claim that in the
process of such perturbation the rotation number must vary. Indeed
if the rotation number of $L$ and $\tilde{L}$ were the same
irrational number, it would imply that two different minimal
geodesics, coming from $L$ and $\tilde{L}$, with the same irrational
rotation number intersect. But this is impossible. Moreover, since
the rotation number is a continuous function,  then this claim
implies that it can be made rational by such a perturbation.
\end{proof}
\section{Separatrix chains}
\label{separatrix} In this section we prove theorems \ref{zones} and
\ref{chains}. We shall use heavily the ordering properties of
minimal geodesics on $\mathbb{T}^2$. The following lemma enables to
repeat the proof of theorem \ref {genus} for the case of theorem
\ref{zones}.
\begin{lemma}
Let  $N$ be a domain $T_1\mathbb{T}^2$ bounded by two disjoint
graphs $L_1,L_2$ invariant under the geodesic flow. Then for any
point $x\in\mathbb{T}^2$ there exist uncountably many rays, such
that the corresponding half-orbits lie in $N$.
\end{lemma}
\begin{proof}[Proof of the lemma]
It follows from the ordering properties of minimal geodesic on the
covering plane. Any two minimal geodesic of different slopes
intersect exactly at one point. Moreover, two different minimal
geodesics of the same slope can intersect only if their slope is
rational and the geodesics are the pair of heteroclinic connections.
Using this fact and the fact that the orbits lying on $L_1, L_2$ (by
theorem \ref{birkhoff}) are minimal, one can conclude that the
rotation numbers of $L_1, L_2$ must be different. Indeed, otherwise
$L_1, L_2$ would have the same rational rotation number and both
would contain minimal periodic geodesics. This would contradict
their disjointness. Therefore, all minimal geodesics and rays
starting at $x$ with slopes in the interval between the rotation
numbers of $L_1$ and $L_2$ lie also in between the torii $L_1, L_2$.
There are uncountably many of them. This completes the proof of the
lemma.
\end {proof}

\begin{proof}[Proof of theorem \ref{chains}]
Let $F$ be a non-constant integral of the geodesic flow. Introduce
any Riemannian metric on $T_1\mathbb{T}^2$ and the vector field
$v=\nabla F /|\nabla F|^2$, defined on an open set of regular points
of $F$. Denote by $v^t$ the flow of $v$. For every regular point
point $x$ there exists a maximal finite open interval of existence
of trajectory $v^t(x)$. For any regular invariant torus $L$ let
$(\alpha(L), \beta(L))$  be the maximal open interval of existence
of $v^t(x)$ for all $x\in L$. At the moments $\alpha(L), \beta(L)$
some points of $L$ tend to the set of critical points. Obviously the
torii $L_t:=v^t(L)$ are also regular invariant torii for any , $t\in
(\alpha(L), \beta(L)).$
 Denote by $B(L)$ the domain swept by $L_t$,
 $$B(L)=\bigcup_{t\in
(\alpha(L), \beta(L))} L_t .$$ Assume now that $L$ is a regular
section. Then it follows from theorem \ref{birkhoff} that all $L_t$
are uniformly $K$-Lipshits. Therefore it follows from Arzela-Ascoli
theorem that there exist uniform limits which are also $K$-Lipshitz
torii:
$$L_{\alpha}=\lim_{t\rightarrow \alpha(L)}L_t,\quad L_{\beta}=\lim_{t\rightarrow
\beta(L)}L_t.$$ Both of the limits are singular invariant torii.
Denote by $B$  the  open set of all regular invariant sections, i.e
$B=\bigcup_L B(L).$ Assume from now on that the integral $F$
satisfies condition $\aleph$ , then by theorem \ref {torus} the set
$B$ is not empty. Fix a regular section $L_0$ in $B$ for the rest of
the proof. Fix an orientation on the fibres. It determines the order
on $B-L_0$. For any regular section $L$ disjoint from $L_0$, there
are two regions between them, $[L_0,L]$ and $[L,L_0]$, in accordance
with orientation of the fibres. Let us describe the complement to
$B$. For any point $P\in T_1\mathbb{T}^2-B$. Define
$$L_+= \inf \{L: L\in B, L\neq L_0 ,
P\in [L_0,L]\},$$ $$L_-=sup\{L:L\in B, L\neq L_0, P \in [L,
L_0]\}.$$ By Arzela-Ascoli theorem these are $K$-Lipshits invariant
sections which are singular and $$L_+\subset [L_-,L_0],\quad
L_-\subset[L_0,L_+].$$ It follows by theorem \ref {zones} that
$L_+,L_-$ cannot be disjoint. It may happen that they coincide, in
this case we are done. In the other case, they have a nontrivial
intersection. Then the rotation number for both of them is the same,
since they have common orbits, and moreover, it must be rational,
because any two minimal geodesics of the same irrational rotation
number can not cross each other. Therefore the intersection $L_+\cap
L_-$ consists of periodic minimal geodesics. It is easy to see that
in fact $L_+\cap L_-$ coincides with the set of all minimal periodic
orbits of this rational type. This completes the proof.

\end{proof}
\section{Remarks on critical points of $F$.}
\label{critical} This section contain some simple facts about
critical points of $F$, for any smooth $F$ (not necessarily
satisfying condition $\aleph$). Let $\gamma$ be a minimal periodic
geodesic on $\Sigma$. We shall say it is deformable if it has a
neighborhood filled by minimal geodesics of the same type. The
following fact is simple: The trajectory of the geodesic flow
corresponding to a non-deformable minimal geodesic consist of
critical points of $F$. This is because, otherwise one could move
the closed orbit in $T_1\Sigma$ by the $f^t$ of $F$ and get a family
of closed orbits nearby. All of them project to closed minimal
geodesics. If for a given periodic type not all minimal geodesics
are closed then there always exist non-deformable minimal geodesic
of this type. For example if $p>1$ for every periodic type there
exists a non-deformable minimal geodesic. For $p=1$, given a
periodic type there exist a non-deformable geodesic of a given type
or there is an invariant torus, consisting of minimal periodic
orbits. Having enough non-deformable minimal geodesics one can take
their limits to get geodesics of other non-periodic types such that
their orbits in $T_1\Sigma$ also consist of critical points of $F$.
In such a way one proves that for $p>1$ and for any not periodic
type there always exist a minimal geodesic of this type such that
the corresponding orbit consists of critical points (These are so
called boundary geodesics in terminology of Morse \cite{Morse}).
Probably the same is true for all recurrent orbits for non-trivial
Aubry-Mather sets in the case of $\Sigma=\mathbb{T}^2$. We don't
know however how this information can be used further.


\end{document}